\documentclass[12pt]{amsart}
\usepackage{graphicx}

\newtheorem{thm}{Theorem}[section]

\newtheorem{ex}[thm]{Example}

\theoremstyle{definition}
\newtheorem{defn}[thm]{Definition}
\theoremstyle{remark}

\numberwithin{equation}{section}

\def\v{\varphi}

\def\l{\lambda}

\def\o{\omega}

\def\G{\Gamma}
\def\O{\Omega}
\def\L{\Lambda}

\def\cal#1{\mathcal{#1}}

\begin{document}

\title[BANACH LIE ALGEBROIDS]{BANACH LIE ALGEBROIDS}%
\author{Mihai ANASTASIEI}%
\address{Faculty of Mathematics,
"Al.I. Cuza" University, Bd. Carol I, Iasi,700506, Romania}%
\email{anastas@uaic.ro}%


\begin{abstract}
First, we extend the notion of second order differential equations
(SODE) on a smooth manifold to anchored Banach vector bundles.
Then we define the Banach Lie algebroids as Lie algebroids
structures modeled on anchored Banach vector bundles and prove
that they form a category.
\end{abstract}
\maketitle
\section*{Introduction}

Lie algebroids are related to many areas of geometry and became
recently an object of extensive studies. See [5] for basic
definitions, examples and  references. In 1996 A. Weinstein [6]
proposed some applications of the Lie algebroids in Analytical
Mechanics. New theoretical developments followed. See the survey
[4] by M. de Leon, J. C. Marrero and E. Martinez about Mechanics on
Lie algebroids.

In [1], we gave a construction of a semispray associated to a
regular Lagrangian on a Lie algebroid.

In this paper, we consider the notion of Lie algebroid in the
category of Banach vector bundles, that is vector bundles over
smooth Banach manifolds whose type fibres are Banach spaces. Such
a Banach vector bundle over base $M$ is called anchored if there
exists a morphism from it to the tangent bundle $TM$. First, we
extend the usual notion of second order differential equations
(SODE) to anchored Banach vector bundles and we show that if a
Banach vector bundle admits a homogeneous SODE it is necessarily
anchored. Then we define the Banach Lie algebroids as Lie
algebroid structures modeled on anchored Banach vector bundles.
In our setting only one from three equivalent definitions of a
morphism of Lie algebroids is working. Using it we show that the
Banach Lie algebroids form a category.

\section{Anchored Banach vector bundles}

Let $M$ be a smooth i.e. $C^{\infty}$ Banach manifold modeled on
Banach space $\mathbf{M}$ and let $\pi:E\to M$ be a Banach vector
bundle whose type fiber is a Banach space $\mathbf{E}$. We denote
by $\tau:TM\to M$ the tangent bundle of $M$.
\begin{defn}
We say that the vector bundle $\pi:E\to M$ is an anchored vector
bundle if there exists a vector bundle  morphism  $\rho:E\to TM$.
The morphism $\rho$ will be called the anchor map.
\end{defn}

Let $\cal{F}(M)$ be the ring of smooth real functions on $M$.
We denote  by $\G(E)$ the $\cal{F}(M)$-module of smooth sections
in the vector bundle $(E,\pi,M)$ and by $\cal{X}(M)$ the module of
smooth sections in the tangent bundle of $M$ (vector fields on
$M$).

The vector bundle morphism $\rho$ induces an $\cal{F}(M)$-module
morphism which will be denoted also by $\rho:\G(E)\to \cal{X}(M)$,
$\rho(s)(x) = \rho(s(x)),$ $x\in M, s\in\G(E)$.

Let $\{(U,\v), (V,\psi), \ldots\}$ be an atlas on $M$. Restricting
$U,V$ if necessary we may choose a vector bundle atlas
$\{(\pi^{-1}(U),\overline{\v}), (\pi^{-1}(V), \overline{\psi}),
\ldots \}$ with $\overline{\v}:\pi^{-1}(U)\to U\times \mathbb{E}$
given by $\overline{\v}(u) = (\pi(u), \overline{\v}_{\pi(u)}),$
where $\overline{\v}_{\pi(u)}:E_{\pi(u)}\to \mathbb{E}$ is a
toplinear isomorphism. Here $E_{\pi(u)}$ is the fiber of
$(E,\pi,M)$ in $u\in E.$ The given atlas on $M$ together with a
vector bundle atlas induce a smooth atlas $\{(\pi^{-1}(U),\phi),
(\pi^{-1}(U),\psi),\ldots \}$ on $E$ such that $E$ becomes a
Banach manifold modeled on the Banach space $\mathbb{M}\times
\mathbb{E}$. The map $\phi:\pi^{-1}(U)\to \v(U)\times \mathbb{E}$
is given by
$$
\phi(u) =
(\v(\pi(u)),\overline{\v}_{\pi(u)}(u)),\;\; u \in E.
$$
For a section $s:U \to \pi^{-1}(U)$, its local representation
$\phi\circ s\circ \v^{-1}:\v(U)\to \v(U)\times \mathbb{E}$ given
by $(\phi\circ s\circ \v^{-1})(\v(x)) = (\v \pi(s(x)),
\overline{\v}_{\pi(s(x))}(s(x))=(\v(x), \overline{\v}_{x}(s(x)))$
is completely determined by the map $s_{\v}:\v(U)\to \mathbb{E}$
given by $s_{\v}(\v(s)) = \overline{\v}_{x}(s(x))$ which will be
called the local representative (shortly l.r.) of $s$. On $U\cap
V$ we may speak also of the l.r. $s_{\psi}$ of a section $s:U\cup
V\to \pi^{-1}(U\cap V)$ given by $s_{\psi}(\psi(x)) =
\overline{\psi}_{x}(s(x))$. It is clear that we have
\begin{equation}
s_{\psi}(\psi(x)) = \overline{\psi}_{x}\circ
\overline{\phi}_{x}^{-1}(s_{\v}(\v(x))),\;\; x\in U\cup V.
\end{equation}
For a vector field $X:U \to \tau^{-1}(U)$ we have a l.r.
$X_{\v}:\v(U)\to \mathbb{M}$ and on $U\cap V$ we have also a l.r.
$X_{\psi}$ and one holds
\begin{equation}
X_{\psi}(\psi(x)) = d(\psi\circ
\v^{-1})(\v(x))(X_{\v}(\v(x))),\;\; x\in U\cap V,
\end{equation}
where $d$ means Frechet differentiation.

Locally, $\rho$ reduces to a morphism $U\times \mathbb{E}\to
U\times \mathbb{M}$, $(x,v)\to (x,\rho_{U}(x)v)$ with
$\rho_{U}(x)\in L(\mathbb{E},\mathbb{M})$, the space of continuous
linear maps from $\mathbf{E}$ to $\mathbf{M}$. We call
$\rho_{U}(x)$ the l.r. of $\rho$. On overlaps of local charts one
easily gets
\begin{equation}
\rho_{V}(x)\circ \overline{\psi}_{x}\circ \overline{\v}_{x}^{-1} =
d(\psi\circ \v^{-1})(\v(x))\circ \rho_{U}(x),\;\; x\in U\cap V
\end{equation}

\begin{ex}
\begin{itemize}
\item[1.] The tangent bundle of $M$ is trivially anchored vector
bundle with $\rho = I$ (identity).

\item[2.] Let $A$ be a tensor field of type $(1,1)$ on $M$. It is
regarded as a section of the bundle of linear mappings
$L(TM,TM)\to M$ and also as a morphism $A :TM\to TM$. In the other
words, $A$ may be thought as an anchor map.

\item[3.] Any subbundle of $TM$ is an anchored vector bundle with
the anchor the inclusion map in $TM$.

\item[4.] Let $\pi:E \to M$ be only a submersion. The subspaces
$V_uE = \pi ^{-1}(x),\pi(x) =u$ of $TE$ over $E$ denoted by $VE$
form a subbundle  called the vertical subbundle. By Example 3)
this is an anchored Banach vector bundle.
\end{itemize}
\end{ex}

The anchored vector bundles over the same base $M$ form a
category. The objects are the pairs $(E,\rho_{E})$ with $\rho_{E}$
the anchor of $E$ and a morphism $f: (E,\rho_{E})\to (F,
\rho_{F})$ is a vector bundle morphism $f:E\to F$ which verifies
the condition $\rho_{F}\circ f = \rho_{E}$.

\section{Semisprays in an anchored vector bundle}

Let $(E,\pi, M)$ be an anchored vector bundle with the anchor map
$\rho$ and let $\pi_{*} :TE\to TM$ be the differential (tangent
map) of $\pi.$

We denote by $\tau_{E}:TE \to E$ the tangent bundle of $E$ .

\begin{defn}
A section $S:E\to TE$ will be called a semispray if

\begin{itemize}
\item[(i)] $\tau_{E}\circ S =$ identity on $E$,

\item[(ii)] $\pi_{\ast}\circ S = \rho.$
\end{itemize}
\end{defn}

The condition (i) says that $S$ is a vector field on $E$. The
condition (ii) can be written also in the form
$$
\pi_{*,u}(S(u)) = \rho(u) = (\rho \circ \tau_{E})(S(u)),\;\; u \in
E.
$$
When $E = TM$ and $\rho =$ identity on $TM$, $S$ is simultaneously
a vector field on $TM$ and a section in the vector bundle
$\pi_{*}: TTM\to TM$ i.e. it is a second-order vector field on $M$
in terminology from [2, p.96]. Such a vector field is frequently
called a second order differential equation (SODE) on $M$ or a
semispray .

As we will see below, in our context $S$ is no more related to a
second order differential equation on $M$ and so the corresponding
terminology is inadequate.

Let $c:J \to E$ for $\circ \in J \subset \mathbb{R}$ a curve on
$E$. The differential of $c$ is $c_{*}:J \times \mathbb{R}\to TE$
and using $\imath:J \to J \times \mathbb{R},$ $t\to (t,1),$ $t\in
J$ we set $c^{\prime}(t) =c_{*}\circ \imath.$

It is clear that $\pi\circ c$ is a curve on $M$ and we have that
$(\pi\circ c)^{\prime}(t) = \pi_{*,c(t)}\circ c^{\prime}(t).$

\begin{defn}
A curve $c$ on $E$ will be called admissible if $(\pi\circ
c)^{\prime}(t) = \rho (c(t)),$ $\forall t\in J.$
\end{defn}

Locally, if $c:J \to \v(U)\times E,$ $t\to (x(t), w(t))$ then
$\pi\circ c:J \to \v(U)$ is $t\to x(t)$, $t\in J$ and it follows
that $x$ is an admissible curve if and only if

\begin{equation}
\frac{dx}{dt} = \rho_{U}(x(t))w(t),\;\; t\in J
\end{equation}

\begin{thm}
A vector field $S$ on $E$ is a semispray if and only if all its
integral curves are admissible curves.
\end{thm}

\begin{proof}
Let $S$ be a semispray. A curve $c:J \to E$ is an integral curve
of $S$ if $c^{\prime}(t) = S(c(t))$. It follows $\pi_{*}\circ
c^{\prime}(t) = (\pi_{*}\circ S)(c(t))$ or $(\pi \circ
c)^{\prime}(t) = \rho(c(t))$, that is $c$ is an admissible curve.
Conversely, let $S$ be a vector field on $E$ whose integral curves
are admissible. For every $u\in E$ there exists an unique integral
curve $c:J \to E$ of $S$ such that $c(0) = u$ and $c^{\prime}(0) =
S(u)$. We have $\pi_{*}\circ c^{\prime}(0) = (\pi_{*}\circ S)(u)$,
$(\pi \circ c)^{\prime}(0) = (\pi_{*}\circ S)(u)$ and
$\pi_{*}\circ S = \rho(u)$ since $c$ is admissible.
\end{proof}

 We restrict to a local chart $(U,\v)$ on $M.$ Then $TU
\simeq \v(U)\times \mathbb{M}$, $E_{|U} \simeq \v(U)\times
\mathbb{E}$ and $TE_{|U} \simeq (\v (U)\times \mathbb{E})\times
\mathbb{M} \times \mathbb{E} $.

The l.r. of a vector field on $E$ is $S_{\v}:\v(U)\times
\mathbb{E} \to \v(U)\times \mathbb{E} \times \mathbb{M}\times
\mathbb{E}$, $S_{\v}(x,u) = (x,u,S_{\v}^{1}(x,u),
S_{\v}^{2}(x,u))$. As l.r. of $\pi_{*}$ is $\v(U)\times
\mathbb{E}\times \mathbb{M}\times \mathbb{E}\to \v(U)\times
\mathbb{M}$, $(x,u,y,v)\to (x,y)$ the condition $\pi_{*}\circ S =
\rho$ translates to $S_{\v}^{1}(x,u) = (x,\rho_{U}(x)u)$. We set
for convenience $S_{\v}^{2}(x,u) = -2 G_{\v}(x,u)$ and so the l.r.
of a semispray for the anchored vector bundle $(E,\pi, M)$ with
the anchor $\rho$ is given as follows:
\begin{equation}
S_{\v}(x,u) = (x,u,\rho_{U}(x)u, - 2G_{\v}(x,u)).
\end{equation}

Let $(V,\psi)$ be another local chart and let us set $h = \psi
\circ \v^{-1}: \v(U\cap V)\to \psi(U\cap V)$. Then $h_{*}:\v
(U\cup V)\times \mathbb{M}\to \psi(U\cap V)\times \mathbb{M}$ is
given by $(x,v)\to (x,dh(x)(v))$, $x\in \v(U\cup V), v\in
\mathbb{M}$.

Let us denote by $H:\v (U\cap V)\times \mathbb{E}\to \psi(U\cap
V)\times \mathbb{E}$ the map given by $H(x,u) = (h(x), M(x)u)$,
where $M(x) = \overline{\psi}_{x}\circ \overline{\v}_{x}^{-1}\in
L(\mathbb{E},\mathbb{E})$. Then $H_{*}$ is locally given as the
pair $(H,H^{\prime})$: $\v(U\cup V)\times \mathbb{E}\times
\mathbb{M}\times \mathbb{E}\to \psi(U\cup V)\times
\mathbb{E}\times \mathbb{M}\times \mathbb{E}$, where the
derivative $H^{\prime}(x,u)$ is given by the Jacobian matrix
operating on the column vector $^{t}(y,w)$ with $y\in \mathbb{M}$
and $w\in \mathbb{E}$. Thus $(H,H^{\prime})$ takes the form
$(x,u,y,v)\to (h(x), M(x)u, h^{\prime}(x)y, M^{\prime}(x)(y)(u) +
M(x)v)$ with prime being denoted the Frechet derivative.

If $S_{\psi}$ is l.r. of $S$ in the chart $(V,\psi)$, necessarily
we have $(H,H^{\prime})\circ S_{\v} = S_{\psi}$ with
$S_{\psi}(x,u) = (h(x), M(x)u, \rho_{U}(h(x))M(x)u, -
2G_{\psi}(h(x), M(x)u))$.

Computing $(H,H^{\prime})\circ S_{\v}$ and identifying with
$S_{\psi}$ one finds
\begin{eqnarray}
\rho_{V}(h(x))M(x)(u) &=& h^{\prime}(x)\rho_{U}(x)(u)\nonumber\\
G_{\psi}(h(x), M(x)u)&=& M(x)G_{\v}(x,u) -
\frac{1}{2}M^{\prime}(x)(\rho_{U}(x)u)u.
\end{eqnarray}
The first equation (2.3) is just (1.3) and the second provides the
connection between the l.r. $G_{\v}$ and $G_{\psi}$ on overlaps.
We have

\begin{thm}
A vector field $S$ on $E$ is a semispray if and only if it has
l.r. $S_{\v}$ in the form (2.2) and the functions involved in
(2.2) satisfy (2.3) on overlaps of local charts.
\end{thm}

\begin{proof}
The "if part" was proven in the above. The converse is
obvious.
\end{proof}

We denote by $h_{\l}:E \to E$, $h_{\l}(u_{x}) =\l u_{x}$, $\l \in
\mathbb{R}$, $\l >0, x\in M,$ the homothety of factor $\l.$

\begin{defn}
We say that a semispray $S$ is a spray if the following equality
holds
\begin{equation}
S\circ h_{\l} = \l (h_{\l})_{*}\circ S.
\end{equation}
Locally, (2.4) is equivalent to
\begin{equation}
G_{\v}(x, \l v) = \l^{2}G_{\v}(x,v),\;\; (x,v)\in U \times
\mathbb{E}.
\end{equation}
\end{defn}

Indeed, $(S\circ h_{\l})(u) = S(\l u) = (x, \l v, \rho_{U}(\l v),
- 2G_{\v}(x,\l v)$ and $\l(h_{\l})_{*}S(u) = (x, \l v, \l
\rho_{U}(v), - 2\l^{2}G_{\v}(x,\l v))$. Since $\rho_{U}$ is a
linear mapping, (2.4) implies (2.5) and conversely.

We look at (2.5). If we fix $x\in U$ and omit the index $\v$ we
get a mapping $G:\mathbb{E}\to \mathbb{E}$ that verifies $G(\l v)
= \l^{r}G(v)$ for all $\l>0$ and $r=2.$ We say that such a map is
positively homogeneous of degree $r$.

For such mapping the following Euler type theorem holds.

\begin{thm}
Suppose that a mapping $G: \mathbb{E}\to \mathbb{E}$ is
differentiable away from the origin of $\mathbb{E}$. Then the
following two statements are equivalent:
\begin{itemize}
\item[(i)] $G$ is positively homogeneous of degree $r$,

\item[(ii)] $dG_{v} (v) = r G(v)$, for all $v\in
\mathbb{E}\backslash \{0\}$.
\end{itemize}
\end{thm}

\begin{proof}
Suppose (i) holds. Fix $v\in \mathbb{E}$ and differentiate the
equation $G(\l y) = \l^{r}G(v)$ with respect to the parameter
$\l.$ We get $dG_{\l v} (\l v) = r\l^{r-1}G(v)$ and for $\l=1,$
$dG_{v}(v) = rG(v),$ that is (ii) holds.

Conversely, suppose (ii), fix $v$ and consider the mapping $\l \to
G(\l v)$ with $\l>0.$ By the chain rule, we have $\frac{dG(\l v
)}{d\l} = dG_{\l v}(v) = \frac{1}{\l}dG_{\l v}(\l v) =
\frac{r}{\l}G(\l v)$, that is the mapping $\l\to G(\l v)$ is a
solution of the differential equation:
$$
\frac{d}{d\l}G(\l v) - \frac{r}{\l}G(\l v) = 0.
$$
The integrating factor $\frac{1}{\l^{r}}$ then gives $G(\l v) =
\l^{r}C,$ where $C$ is depending on our fixed $v$. Setting $\l=1,$
we get $C = G(v)$ and so $G(\l v) = \l^{r}G(v),$ that is (i)
holds, q.e.d.
\end{proof}

The proof of Theorem 2.6 shows also that if $G:\mathbb{E}\to
\mathbb{E}$ is of class $C^{1}$ on $\mathbb{E}$ and positively
homogeneous of degree 1, then it is linear and $G(v) = dG_{v}(v)$.
Moreover, if $G$ is $C^{2}$ on $\mathbb{E}$ and is positively
homogeneous of degree 2, then it is quadratic, that is $2G(v) =
d_{v}^{2}G(v,v).$

Returning to the (2.5) we note that if $G_{\v}$ is of class
$C^{2}$ in the points $(x,0)$, then it is quadratic in $v$.

Thus $S$ satisfying (2.4) reduces to a quadratic spray. For
avoiding this reduction we have to delete from $E$ the image of
the null section in the vector bundle $\pi: E \to M.$

Now, we show that if for a vector bundle $E\to M$ there exists a
vector field $S_{0}$ on $E$ that satisfies (2.4) then $\pi:E \to
M$ is an anchored vector bundle and $S_{0}$ is a spray.

Let be $S_{0}(x,v) = (x,v, S_{01}(x,v), S_{0,2}(x,v))$ in a local
chart on $E$. Then $S_{0}(h_{\l}u) = S_{0}(x,\l v) = (x, \l v,
S_{01}(x, \l v ), S_{02}(x,\l v))$ and $(h_{\l})_{*}S_{0}(u) = (x,
\l v, S_{01}(x, v), \l S_{02}(x,v))$. The condition (2.4) implies
$S_{01}(x,\l v) = \l S_{01}(x,v)$ and $S_{02}(x, \l v) =
\l^{2}S_{02}(x,v)$. It follows that $S_{01}$ is a linear map with
respect to $v$. Hence we may put $S_{01}(x,v) = \rho_{U}(x)v,$
$\rho_{U}(x): \mathbb{E}\to \mathbb{M}$. Using $\{\rho_{U}(x),
x\in M \}$ one defines a morphism $\rho:E \to TM.$ Thus $E\to M$
is an anchored vector bundle. As $(\pi_{*}S_{0})(u) = (x,
S_{01}(x,v)) = (x, \rho_{U}(x)v)$ we have $\pi_{*}\circ S_{0}
=\rho$ and as $\tau_{E}\circ S_{0}$=indentity automatically holds
it follows that $S_{0}$ is a spray.

\section{Category of Banach Lie algebroids}

Let $\pi:E \to M$ be an anchored Banach vector bundle with the
anchor $\rho_{E}:E \to TM$ and the induced morphism
$\rho_{E}:\G(E)\to \mathcal{X}(M)$.

Assume there exists defined a bracket $[,]_{E}$ on the space
$\G(E)$ that provides a structure of real Lie algebra on $\G(E)$.

\begin{defn}
The triplet $(E, \rho_{E}, [,]_{E})$ is called a Banach Lie
algebroid if
\begin{itemize}
\item[(i)] $\rho: (\G(E), [,]_{E})\to (\mathcal{X}(M), [,])$ is a
Lie algebra homomorphism and

\item[(ii)] $[s_{1}, s_{2}]_{E} = f[s_{1}, s_{2}]_{E}
+\rho_{E}(s_{1})(f)s_{2}$, for every $f\in \mathcal{F}(M)$ and
$s_{1}, s_{2}\in \G(E)$.
\end{itemize}
\end{defn}

\begin{ex}
\begin{itemize}
\item[1.] The tangent bundle $\tau: TM\to M$ is a Banach Lie
algebroid with the anchor the identity map and the usual Lie
bracket of vector fields on $M$.

\item[2.] For any submersion $\pi:E \to M$, the vertical bundle
$VE$ over $E$ is an anchored Banach vector bundle. As the Lie
bracket of two vertical vector fields is again a vertical vector
field it follows that $(VE, i, [,]_{VE})$, where $i:VE \to TE$ is
the inclusion map, is a Banach Lie algebroid. This applies, in
particular, to any Banach vector bundle $\pi:E \to M.$
\end{itemize}
\end{ex}

Let $\Omega^{q}(E):= \G(\L^{q}A^{*})$ be the $\mathcal{F}(M)-$
module of differential forms of degree $q$. In particular,
$\O^{q}(TM)$ will be denoted by $\O^{q}(M)$. The differential
operator $d_{E}:\O^{q}(E)\to \O^{q+1}(E)$ is given by the formula
\begin{eqnarray}
&&(d_{A} \o)(s_{0}, \ldots, s_{q})  =
\sum_{i=0,\ldots,n}(-1)^{i}\rho_{E}(s_{i})\o(s_{0}, \ldots,
\widehat{s}_{i}, \ldots, s_{q})\nonumber\\&&+ \sum_{0\leq i< j
\leq q} (-1)^{i+j}\o ([s_{i},s_{j}]_{E}), s_{0}, \ldots
\widehat{s}_{i}, \ldots , \widehat{s}_{j}, \ldots, s_{q}
\end{eqnarray}
for $s_{1}, \ldots, s_{q}\in \G(E)$, where hat over a symbol means
that that symbol must be deleted.

For Lie algebroids constructed on vector bundles with finite
dimensional fibres there exist three different but equivalent
notions of morphisms.

For Banach Lie algebroids only one of them is working. We give it
here. For a detailed discussion on Lie algebroids morphisms see
[3]. Let $(E^{\prime}, \pi^{\prime}, M)$ be a Banach vector bundle
and $(E^{\prime}, \rho_{E^{\prime}}, [,]_{E^{\prime}})$ a Banach
Lie algebroids based on it.

\begin{defn}
A vector bundle morphism $f:E \to E^{\prime}$ over $f_{0}:M\to
M^{\prime}$ is a morphism of the Banach Lie algebroids $(E,
\rho_{E}, [,]_{E})$ and $(E^{\prime}, \rho_{E^{\prime}},
[,]_{E^{\prime}})$ if the map induced on forms
$f^{*}:\O^{q}(E^{\prime})\to \O^{q}(E)$ defined by
$(f^{*}\o^{\prime})_{x}(s_{1}, \ldots, s_{q}) =
\o^{\prime}_{f_{0}(x)}(fs_{1}, \ldots, fs_{q})$, $s_{1}, \ldots,
s_{2}\in \G(E)$ commutes with the differential i.e.
\begin{equation}
d_E \circ f^{*} = f^{*}\circ d_{E}.
\end{equation}
\end{defn}

Using this definition it is easy to prove

\begin{thm}
The Banach Lie algebroids with the morphisms defined in the above,
form a category.
\end{thm}


\begin{thebibliography}{11}

\bibitem{1} Anastasiei, M., Geometry of Lagrangians and semispray on
Lie algebroids. BSG Proceedings 13, Geometry Balkan Press, 2006,
p.10-17.

\bibitem{2} Lang, S., Fundamentals of Differential
Geometry,Graduate Texts in Mathematics 191, Springer, 1999.

\bibitem{3} Higgins, P,J., Mackenzie, K., Algebraic constructions in
the category of Lie algebroids. J. Algebra, 129 (1990),no.1,
194-230.

\bibitem{4} Leon de, M., Marrero J.C., Martinez E., Lagrangian
submanifolds and dynamics on Lie algebroids, J. Phys.A:
Math.Gen.38(2005), R241-R308


\bibitem{5} Mackenzie K., General theory of Lie grupoids and Lie
algebroids. London Mathematical Society Lecture Note Series,213,
Cambridge University Press, Cambridge, 2005, 501 p.

\bibitem{6} Weinstein A., Lagrangian mechanics and grupoids. Fields
Inst. Comm. 7(1966),207-231


\end{thebibliography}
\end{document}